\documentclass[12pt Times New Roman]{amsart}
\usepackage{amsmath, amsthm, amscd, }
\usepackage{graphicx}
\usepackage{setspace}
\newfont{\aj}{eufm10 at12pt}
\newfont{\ajk}{eufm10 at10pt}
\theoremstyle{plain}
\newtheorem{theorem}{Theorem}[section]
\newtheorem{lemma}[theorem]{Lemma}

\newtheorem{corollary}[theorem]{Corollary}
\theoremstyle{definition}
\newtheorem{definition}[theorem]{Definition}

\newtheorem{remark}[theorem]{Remark}
\numberwithin{equation}{section}
\begin{document}
\title[Various shadowing properties and their equivalent for... ]{ Various shadowing properties and their equivalent for expansive iterated function systems}
\author[Mehdi Fatehi Nia]{ {{\bf {Mehdi Fatehi Nia}}
\\{\tiny{Department of Mathematics, Yazd University, 89195-741 Yazd, Iran}
\\e-mail: fatehiniam@yazd.ac.ir }}}
\maketitle
\vspace{5mm}
\begin{abstract}
In this paper we introduce expansive \emph{iterated function systems}, ( \textbf{IFS}) on a compact metric space then  various  shadowing properties and their equivalence are considered for expansive  \textbf{IFS}.
\end{abstract}
\emph{keywords}: {Expansive IFS, pseudo orbit, shadowing, continuous shadowing, limit shadowing, Lipschitz shadowing.}\\
\emph{subjclass}[2010]{ 37C50,37C15}
\section{Introduction}
The notion of shadowing plays  an important role in dynamical systems, specially; in stability theory \cite{[AH],[LKA],[KP]}. Various shadowing properties for expansive maps and their equivalence  have been studied by Lee and Sakai \cite{[LK],[KS]}. More precisely, they prove the following theorems:
\begin{theorem}\cite{[LK]}\label{maina}
 Let $f$ be an expansive homeomorphism on a compact metric space $(X; d)$. Then
the following conditions are mutually equivalent:\\
$(a)~$ $f$ has the shadowing property,\\
$(b) ~f$ has the continuous shadowing property,\\
$(c)$ there is a compatible metric $D$ for $X$ such that $f$ has the Lipschitz shadowing property
with respect to $D$,\\
$(d)~ f$ has the limit shadowing property,\\
$(e)$ there is a compatible metric $D$ for $X$ such that $f$ has the strong shadowing property
with respect to $D$.
\end{theorem}
\begin{theorem}\cite{[KS]}\label{mainb}
Let $f$ be a positively expansive map on a compact metrizable space $X$. Then
the following conditions are mutually equivalent:\\
$(a) ~f$ is an open map,\\
$(b) ~f$ has the shadowing property,\\
$(c)$ there is a metric such that $f$ has the Lipschitz shadowing property,\\
$(d)$ there is a metric such that $f$ has the $s$-limit shadowing property,\\
$(e)$ there is a metric such that $f$ has the strong shadowing property.
\end{theorem}
 In the other hand, iterated function systems( \textbf{IFS}), are used for the construction of deterministic fractals and have found numerous applications, in
particular to image compression and image processing \cite{[B]}. Important notions in
dynamics like attractors, minimality, transitivity, and shadowing can be extended to
IFS (see \cite{[BV],[BVA],[GG],[GG1]}). The authors defined the shadowing property for a parameterized iterated function system and prove that if a parameterized IFS is uniformly expanding ( or contracting), then it has the shadowing property\cite{[GG]}. \\ In this paper we present an approach to shadowing property for iterated function systems. At first, we introduce \emph{expansive iterated function systems} on a compact metric space. Then continuous shadowing, limit shadowing and Lipschitz shadowing properties are defined  for an  \textbf{IFS}, $ \mathcal{F}=\{X; f_{\lambda}|\lambda\in\Lambda\}$  where $\Lambda$ is a nonempty finite set and $f_{\lambda}:X\rightarrow X$ is homeomorphism, for all $\lambda\in\Lambda$. Theorems \ref{ta} and \ref{tb} are the main result of the present work. Actually in these theorems we prove that  the limit shadowing property, the Lipschitz shadowing property are all equivalent to the shadowing property for expansive \textbf{IFS} on a compact metric space. The method is essentially the same as that used in \cite{[LK],[Re],[KS]}. Finally, we introduce the strong expansive \textbf{IFS} and show that for a strong expansive \textbf{IFS} the continuous shadowing property and the shadowing property are equivalent.
\section{preliminaries}
In this section, we give some definitions and notations as well as some
preliminary results that are needed in the sequel.
Let $(X,d)$ be a complete metric space. Let us recall that an \emph{ Iterated Function System(IFS)} $\mathcal{F}=\{X; f_{\lambda}|\lambda\in\Lambda\}$ is any family of continuous mappings $f_{\lambda}:X\rightarrow X,~\lambda\in \Lambda$, where $\Lambda$ is a finite nonempty set (see\cite{[GG]}).\\ Let  $\Lambda^{\mathbb{Z}}$ denote the set of all infinite sequences $\{\lambda_{i}\}_{i\in \mathbb{Z}}$ of symbols belonging to $\Lambda$. A typical element of $\Lambda^{\mathbb{Z}}$
can be denoted as $\sigma= \{...,\lambda_{-1},\lambda_{0},\lambda_{1},...\}$ and we use the shorted notation
$$\mathcal{F}_{\sigma_{0}}=id,$$$$\mathcal{F}_{\sigma_{n}}=f_{\lambda_{n-1}} o f_{\lambda_{n-2}}o ...o f_{\lambda_{0}},$$$$
\mathcal{F}_{\sigma_{-n}}=f^{-1}_{\lambda_{-n}} o f^{-1}_{\lambda_{-(n-1)}}o ...o f^{-1}_{\lambda_{-1}}.$$
Please note that if  $f_{\lambda}$ is a homeomorphism map for all   $\lambda\in\Lambda$, then for every $n\in\mathbb{Z}$ and  $\sigma\in \Lambda^{\mathbb{Z}}$, $\mathcal{F}_{\sigma_{n}}$ is a homeomorphism map on $X$. \\A sequence $\{x_{n}\}_{n\in \mathbb{Z}}$ in $X$ is
called an orbit of the \textbf{IFS} $\mathcal{F}$ if there exist $\sigma\in \Lambda^{\mathbb{Z}}$ such that $x_{n+1}=f_{\lambda_{n}}(x_{n})$,
for each $\lambda_{n}\in \sigma$.\\  
The \textbf{IFS}
$ \mathcal{F}=\{X; f_{\lambda}|\lambda\in\Lambda\}$ is \emph{uniformly expanding} if there exists
$$\beta= sup_{\lambda\in\Lambda} sup_{x\neq y}\frac{d(f_{\lambda}(x),f_{\lambda}(y))}{d(x,y)} $$and this number called also the \emph{expanding  ratio},
is greater than one \cite{[GG]}.
  \\
We say that $\mathcal{F}$ is  expansive if there exist a $e> 0$ such that for every  arbitrary $\sigma\in \Lambda^{\mathbb{Z}}$,
 $d(\mathcal{F}_{\sigma_{i}}(x),\mathcal{F}_{\sigma_{i}}(y))<e$, for all $i \in \mathbb{Z}$, implies that $x=y$.
 \begin{remark}\label{rea}
Let  $\mathcal{F}$ be an uniformly expanding \textbf{IFS} and  $\beta>1$ is it's expanding ratio number. Suppose that $\sigma\in \Lambda^{\mathbb{Z}}$ and  $d(\mathcal{F}_{\sigma_{i}}(x),\mathcal{F}_{\sigma_{i}}(y))<1$ for all $i \in \mathbb{Z}$. So, $\beta^{i}d(x,y)<d(\mathcal{F}_{\sigma_{i}}(x),\mathcal{F}_{\sigma_{i}}(y))<1$, for all $i >0$, and consequently $x=y$. Then  uniformly expanding implies the expansivity.
 \end{remark}
 Given $\delta>0$, a sequence $\{x_{i}\}_{i\in \mathbb{Z}}$ in $X$ is called a $\delta-$pseudo orbit of $\mathcal{F}$ if there exist
$\sigma\in\Lambda^{\mathbb{Z}}$ such that for every $\lambda_{i}\in \sigma$, we have $d(x_{i+1},f_{\lambda_{i}}(x_{i}))<\delta$.\\One says that the
 \textbf{IFS} $ \mathcal{F}$ has the \emph{shadowing property } if, given $\epsilon>0$, there exists $\delta>0$ such that for any $\delta-$pseudo
orbit $\{x_{i}\}_{i\in \mathbb{Z}}$ there exist an orbit $\{y_{i}\}_{i\in \mathbb{Z}}$, satisfying the inequality $d(x_{i},y_{i})\leq \epsilon$ for all $i\in \mathbb{Z}$.
In this case one says that the $\{y_{i}\}_{i\in \mathbb{Z}}$ or the point $y_{0}$, $\epsilon-$ shadows the $\delta-$pseudo orbit $\{x_{i}\}_{i\in \mathbb{Z}}$\cite{[GG]}.\\
Please note that if $\Lambda$ is a set with one member then the \textbf{IFS} $ \mathcal{F}$ is an ordinary discrete dynamical system.
In this case the shadowing property for $\mathcal{F}$ is ordinary shadowing property for a discrete dynamical system.
\begin{remark}\label{remaa}
If the \textbf{IFS} $ \mathcal{F}$ is expansive with expansive constant  $e> 0$ and $0<\epsilon<\frac{e}{3}$ then the point $y_{0}$ in the above definition is unique.
\end{remark}
 We say that $\mathcal{F}$ has the \emph{Lipschitz shadowing property} if there are  $L>0$ and $\epsilon_{0}>0$ such that for any $0<\epsilon<\epsilon_{0}$
 and any $\epsilon-$pseudo orbit $\{x_{i}\}_{i \in \mathbb{Z}}$ of $\mathcal{F}$ there exist
 $y\in X$ and $\sigma\in \Lambda^{\mathbb{Z}}$ such that $d(\mathcal{F}_{\sigma_{i}}(y),x_{i})<L\epsilon$, for all $i \in \mathbb{Z}$.\\
We say that $\mathcal{F}$ has the \emph{ limit shadowing property} if:
 for any sequence $\{x_{i}\}_{i \in \mathbb{Z}}$ of points in $X$, if $\lim_{i\rightarrow \pm\infty}d(f_{\lambda_{i}}(x_{i}),x_{i+1})=0$,
 for some $\sigma= \{...,\lambda_{-1},\lambda_{0},\lambda_{1},...\}\in \Lambda^{\mathbb{Z}}$ then there is an orbit $\{y_{i}\}_{i\in \mathbb{Z}}$ such that  $\lim_{i\rightarrow \pm\infty}d(x_{i},y_{i})=0$\\
 Let $X^{Z}$ be the set of all sequences $\{x_{i}\}_{i\in \mathbb{Z}}$ of points in $X$ and let $\widetilde{d}$ be the metric on $X^{Z}$
defined  by
$$ \widetilde{d}(\{x_{i}\}_{i\in \mathbb{Z}}; \{y_{i}\}_{i\in \mathbb{Z}}) = sup_{i\in \mathbb{Z}}
\frac{d(x_{i}; y_{i})}{2^{\mid i\mid}},$$for $\{x_{i}\}_{i\in \mathbb{Z}}, \{y_{i}\}_{i\in \mathbb{Z}}\in X^{Z}$. Let $p(\mathcal{F},\delta)$ be the set of all $\delta$-pseudo-orbits $(\delta > 0)$ of $\mathcal{F}$  with the
subspace topology of $X^{Z}$ \cite{[LK]}.\\
 We say that $\mathcal{F}$ has the \emph{continuous  shadowing property} if for every $\epsilon>0$, there are a $\delta>0$ and a continuous map $r:p(\mathcal{F},\delta)\rightarrow X$ such that $d(\mathcal{F}_{\sigma_{i}}(r(\texttt{x})),x_{i})<\epsilon$, where \\ $\sigma=\{...,\lambda_{-1},\lambda_{0},\lambda_{1},...\}$, $\texttt{x}=\{x_{i}\}_{i\in \mathbb{Z}}$ and $d(f_{\lambda_{i}}(x_{i}),x_{i+1})<\delta$, for all $i\in \mathbb{Z}$.

 \section{\textbf{ Results}}
By Theorem \ref{mainb}, Sakai showed that any positively expansive open map has the
shadowing property. In this section we introduce open \textbf{IFS} and show that for an  expansive  \textbf{IFS}, the openness;  shadowing property and Lipschitz shadowing property are equivalent.\\
$\mathcal{F}=\{X; f_{\lambda}|\lambda\in\Lambda\} $ is said to be an open \textbf{IFS} if $f_{\lambda}$ is an open map, for all  $\lambda\in\Lambda$.
\begin{definition}\label{defa}\cite{[KS]}
Let $f:X \rightarrow X$ be a continuous map on a compact metric space. We say that $f$
expands small distances if there exist  constants $\delta_{0} > 0$ and $\alpha > 1$
such that $0<d(x,y)<\delta_{0}$ ($x, y \in X$) implies $d(f (x),f (y))>\alpha d(x, y)$.
\end{definition}
We say that  $\mathcal{F}$ expands small distance, if there are constants $\delta_{0}>0$ and $\alpha>1$ such that $d(f_{\lambda}(x),f_{\lambda}(y))>\alpha d(x,y)$ whenever $0<d(x,y)<\delta_{0}~(\lambda\in\Lambda)$
\begin{remark}\label{rea}
Suppose that $\mathcal{F}=\{X; f_{\lambda}|\lambda\in\Lambda\} $ is an expansive \textbf{IFS}, then $f_{\lambda}$ is an expansive function and by Lemma 1. of \cite{[KS]} expand small distance. Let in the proof of Lemma 1. of \cite{[KS]}$$V_{n}=\{(x,y)\in X\times X:d(\mathcal{F}_{\sigma_{i}}(x),\mathcal{F}_{\sigma_{i}}(y))\leq c, for ~all~\sigma\in \Lambda^{\mathbb{Z}} ~and ~all~ \mid i\mid<n\}. $$ So, $\mathcal{F}$ expand small distance.
\end{remark}
To prove Theorem \ref{ta}, we need the following lemma.
\begin{lemma}\label{le1}
Suppose that $\mathcal{F}$ expands small distance with related constants $\alpha>1$ and $\delta_{0}>0$,  then the following are equivalent:
 \\
 $i)~\mathcal{F}$ is an open IFS.\\
 $ii)$ There exists $0<\delta_{1}<\frac{\delta_{0}}{2}$ such that if $d(f_{\lambda}(x),y)<\alpha\delta_{1}$ then $B_{\delta_{1}}(x)\cap f_{\lambda}^{-1}(y)\neq \emptyset $, for all $\lambda\in \Lambda$, where $ B_{\delta_{1}}(x)$ is the neighborhood of $x$ with
radius $\delta_{1}$.
\end{lemma}
\begin{proof}
Since every $f_{\lambda}$ expands small distance then by Lemma 2. of \cite{[KS]}, for every $\lambda\in\Lambda$ the following are equivalent:
 \\
 $i)~f_{\lambda}$ is an open map.\\
 $ii)$ there exists $0<\delta_{\lambda}<\frac{\delta_{0}}{2}$ such that if $d(f_{\lambda}(x),y)<\alpha\delta_{\lambda}$ then $B_{\delta_{\lambda}}(x)\cap f_{\lambda}^{-1}(y)\neq \emptyset $.\\
 Because of the proof of Lemma $1$ in \cite{[CR]} for every $\lambda\in\Lambda$ there exist infinitely $0<\delta<\delta_{\lambda}$ such that  $d(f_{\lambda}(x),y)<\alpha\delta$ implies $B_{\delta_{\lambda}}(x)\cap f_{\lambda}^{-1}(y)\neq \emptyset $. So this sufficient to take\\ $\delta_{1}=\min\{\delta_{\lambda}:\lambda\in\Lambda\}$.
\end{proof}
\begin{theorem}\label{ta}
Under the above assumption, the following conditions are equivalent:\\
$i)~\mathcal{F}$ is an open IFS.\\
$ii)~\mathcal{F}$ has the shadowing property. \\
$iii)~\mathcal{F}$ has the Lipschitz shadowing property.
\end{theorem}
\begin{proof}
$(iii\Rightarrow ii)$ By definitions of the shadowing and Lipschitz shadowing properties this is clear that the Lipschitz shadowing property implies the  shadowing property.\\
$(i\Longrightarrow iii)$ Let $L=\frac{2\alpha}{\alpha-1}=2\Sigma_{k=0}^{\infty}\alpha^{-k}>1$ and fix any $0<\epsilon<\frac{\delta_{1}}{L}$, where $\delta_{1}$ and $\alpha$ be as in Lemma \ref{le1}. Suppose that $\{x_{i}\}_{i\in \mathbb{Z}}$ is an $\epsilon-$pseudo orbit for $ \mathcal{F}$; there is \\$\sigma= \{...,\lambda_{-1},\lambda_{0},\lambda_{1},...\}\in \Lambda^{\mathbb{Z}}$ such that $d(f_{\lambda_{i}}(x_{i}),x_{i+1})<\epsilon$ for all $i\in \mathbb{Z}$.\\ Pick any $i\geq 1$ and put $\alpha_{j}=\Sigma_{k=0}^{j-1}\alpha^{-k}$ for $j\geq 1$. Since $d(f_{\lambda_{i}}(x_{i}),x_{i+1})<\epsilon$ then, by Lemma \ref{le1}, there exists $y_{i-1}^{(i)}\in B_{\frac{\epsilon}{\alpha}}(x_{i-1})$ such that $f_{\lambda_{i-1}}(y_{i-1}^{(i)})=x_{i}$. Thus\\ $d(f_{\lambda_{i-2}}(x_{i-2}),y_{i-1}^{(i)})\leq d(f_{\lambda_{i-2}}(x_{i-2}),x_{i-1})+d(x_{i-1},y_{i-1}^{(i)})<\epsilon+\frac{\epsilon}{\alpha}=\epsilon(1+\frac{1}{\alpha})<\epsilon L$.\\
Hence there exists  $y_{i-2}^{(i)}\in B_{\alpha_{2}\frac{\epsilon}{\alpha}}(x_{i-2})$ such that $f_{\lambda_{i-2}}(y_{i-2}^{(i)})=y_{i-1}^{(i)}$ and so\\ $d(f_{\lambda_{i-2}}(x_{i-3}),y_{i-2}^{(i)})\leq d(f_{\lambda_{i-3}}(x_{i-3}),x_{i-2})+d(x_{i-2},y_{i-2}^{(i)})<\epsilon+\alpha_{2}\frac{\epsilon}{\alpha}<\alpha_{3}\epsilon <\epsilon L$.\\
Because of Lemma \ref{le1} there exists  $y_{i-3}^{(i)}\in B_{\alpha_{3}\frac{\epsilon}{\alpha}}(x_{i-3})$ such that $f_{\lambda_{i-3}}(y_{i-3}^{(i)})=y_{i-2}^{(i)}$. Thus $d(f_{\lambda_{i-3}}(x_{i-4}),y_{i-3}^{(i)})<\alpha_{4}\epsilon<\epsilon L$.\\
Repeating the process, we can find:\\
$y_{0}^{(i)}\in B_{\alpha_{i}\frac{\epsilon}{\alpha}}(x_{0})$ such that $f_{\lambda_{0}}(y_{0}^{(i)})=y_{1}^{(i)}$,\\
$y_{-1}^{(i)}\in B_{\alpha_{i+1}\frac{\epsilon}{\alpha}}(x_{0})$ such that $f_{\lambda_{-1}}(y_{-1}^{(i)})=y_{0}^{(i)}$,\\
$\vdots$
\\
$y_{-i}^{(i)}\in B_{\alpha_{2i}\frac{\epsilon}{\alpha}}(x_{0})$ such that $f_{\lambda_{-i}}(y_{-i}^{(i)})=y_{-i+1}^{(i)}$.\\
Since $X$ is compact, if we let $y_{k}=\lim_{i\rightarrow \infty}y_{k}^{(i)}$, then $f_{\lambda_{k}}(y_{k})=y_{k+1}$ and $d(y_{k},x_{k})<\epsilon L$, for all $k\in \mathbb{Z}$. Therefore $\mathcal{F}$ has the Lipschitz shadowing property.
\\$(ii\Rightarrow i)$.  Since $\mathcal{F}$ has the shadowing property, there exist $0<\delta<\frac{\delta_{0}}{2}$ such that every $\delta\alpha-$pseudo orbit of $\mathcal{F}$ is $\delta_{0}-$shadowed by some point. Now, fix $\nu\in \Lambda$. Consider $x, y\in X$ such that $d(f_{\nu}(x),y)<\delta\alpha$ and define a $\delta\alpha-$pseudo orbit of $\mathcal{F}$ by $x_{0}=x$ and $x_{i}=f_{\nu}^{i-1}(y)~ (i\in\mathbb{Z})$. Then there exists $z\in X$ and $\sigma= \{...,\lambda_{-1},\lambda_{0},\lambda_{1},...\}\in \Lambda^{\mathbb{Z}}$ such that $d(\mathcal{F}_{\sigma_{i}}(z),x_{i})<\delta_{0}$, for all $i\in \mathbb{Z}$. Less of generality; by proof of Theorem 2.2. in \cite{[GG]} and this fact that $x_{i+1}=f_{\nu}(x_{i})~ (i\in\mathbb{Z})$, we can assume that that $\lambda_{i}=\nu$ for all $i\geq 0$. Then $\alpha^{i-1}d(f_{\nu}(z),y)\leq d(f_{\nu}^{i}(z),f_{\nu}^{i-1}(y))\leq \delta_{0}$ for all $~ i\geq 0$, so $f_{\nu}(z)=y$. This implies that $z=f_{\nu}^{-1}(y)$ and $d(x,z)<\delta_{0}$, then $d(x,z)<\frac{d(f_{\nu}(x),f_{\nu}(z))}{\alpha}=\frac{(f_{\nu}(x),y)}{\alpha}<\frac{\delta\alpha}{\alpha}$. Hence $B_{\delta}(x)\cap f_{\nu}^{-1}(y)\neq\emptyset$. So, by Lemma \ref{le1} $\mathcal{F}$ is an open IFS.
\end{proof}
The next theorem is one of the main results of this paper and demonstrates that for an expansive \textbf{IFS}, the limit shadowing property and the  shadowing property are  equivalent.
\begin{theorem}\label{tb}
Let $X$ be a compact metric space and $\mathcal{F}=\{X; f_{\lambda}|\lambda\in\Lambda\} $ be an expansive IFS on $\mathbb{Z}$. The following conditions are equivalent:\\
 $i)~\mathcal{F}$ has the shadowing property,\\
 $ii)~$ there is a compatible metric $D$ for $X$ such that $ ~\mathcal{F}$ has the limit shadowing property with respect to $D$.
\end{theorem}
\begin{proof}
By definitions the assertion $(ii\Rightarrow i)$ is clear.\\

To prove $(i\Rightarrow ii)$, at first we have the following lemmas.
\begin{lemma}
There is a compatible metric $D$ on $X$ and $K\geq 1$ such that
 \[\left\lbrace
  \begin{array}{c l}
 D(f_{\lambda}(x),f_{\lambda}(y))\leq K D(x,y),\\
    D(f^{-1}_{\lambda}(x),f^{-1}_{\lambda}(y))\leq K D(x,y)
 \end{array}\right.\] for any $x,y\in X$ and $\lambda\in \Lambda$.
\end{lemma}
\begin{proof}
 Since $ \mathcal{F}$ is expansive then $f_{\lambda}$ is expansive, for every $\lambda\in \Lambda$. So, by \cite{[KS]} (page 3) for every $\lambda\in \Lambda$ there exists $K_{\lambda}>1$ such that
\[  \left\lbrace
  \begin{array}{c l}
 D(f_{\lambda}(x),f_{\lambda}(y))\leq K_{\lambda} D(x,y),\\
    D(f^{-1}_{\lambda}(x),f^{-1}_{\lambda}(y))\leq K_{\lambda} D(x,y)
 \end{array}
 \right. \] for any $x,y\in X$. Take $K=\max\{K_{\lambda}:\lambda\in \Lambda\}$, the proof is complete.
\end{proof}
 To prove $(i\Rightarrow ii)$ we need to define the local stable set and the local unstable set for an \textbf{IFS}. \\Let $\epsilon>0$,  $\sigma\in \Lambda^{\mathbb{Z}}$ and $x$ be an arbitrary point of $X$ then\\
$W_{\epsilon_{0}}^{s}(x,\sigma)=\{y;d(\mathcal{F}_{\sigma_{n}}(x),\mathcal{F}_{\sigma_{n}}(y))\leq \epsilon,~ \forall n\geq 0\}$,\\
$W_{\epsilon_{0}}^{u}(x,\sigma)=\{y;d(\mathcal{F}_{\sigma_{-n}}(x),\mathcal{F}_{\sigma_{-n}}(y))\leq \epsilon ,~ \forall n >0\}$\\
is said to be the local stable set and the local unstable set of $x$ respect to $\sigma\in \Lambda^{\mathbb{Z}}$.
\begin{lemma}\label{leg}
There exist constants $\epsilon_{0}>0$ and $\eta<1$ such that
\[  \left\lbrace
  \begin{array}{c l}
 D(\mathcal{F}_{\sigma_{i}}(x),\mathcal{F}_{\sigma_{i}}(y))\leq \eta^{i} D(x,y)& \text{if ~$~y\in W_{\epsilon_{0}}^{s}(x,\sigma)$},\\
   D(\mathcal{F}_{\sigma_{-i}}(x),\mathcal{F}_{\sigma_{-i}}(y))\leq \eta^{i} D(x,y)& \text{if ~$~y\in W_{\epsilon_{0}}^{u}(x,\sigma)$}
 \end{array}
 \right. \]
\end{lemma}
\begin{proof}
Since for every $\lambda\in \Lambda$, $f_{\lambda}$ is an expansive map,
To proof the lemma this is sufficient to  in Lemma 1 of  \cite {[Re]} we assume that
$$W_{n}=\{(x,y)\in X\times X:d(\mathcal{F}_{\sigma_{i}}(x),\mathcal{F}_{\sigma_{i}}(y))\leq c, for ~all~ \mid i\mid<n\}. $$
The rest of proof is similar to \cite{[Re]}.
\end{proof}
$(i\Rightarrow ii)$ Let $D$ be the compatible metric for $X$ by the above lemmas. Let $\{x_{i}\}_{i\in\mathbb{Z}}$ be any $\epsilon-$pseudo orbit of $\mathcal{F}$, $(\epsilon\leq \frac{\epsilon_{0}}{2L})$ i.e. $D(f_{\lambda_{i}}(x_{i}),x_{i+1})<\epsilon$, for some \\$\sigma= \{...,\lambda_{-1},\lambda_{0},\lambda_{1},...\}\in \Lambda^{\mathbb{Z}}$. Then by Theorem \ref{ta}, there exists $y\in X$ such that $D(\mathcal{F}_{\sigma_{i}}(y),x_{i})<L\epsilon$ for all $i\in \mathbb{Z}$. Suppose further that $\lim_{i\rightarrow\pm\infty}D(f_{\lambda_{i}}(x_{i}),x_{i+1})=0$. For any $\delta>0 ~(\delta<\frac{\epsilon_{0}}{2L}),$ there exists $I_{\delta}>0$ such that $\mid i\mid>I_{\delta}$ implies that $D(f_{\lambda_{i}}(x_{i}),x_{i+1})<\delta$. Note that
$$\{...,f^{-1}_{\lambda_{I_{\delta}-2}}(f^{-1}_{\lambda_{I_{\delta}-1}}(x_{I_{\delta}})),f^{-1}_{\lambda_{I_{\delta}-1}}(x_{I_{\delta}}),x_{I_{\delta}},x_{I_{\delta}+1},...\}$$
is a $\delta-$pseudo orbit of $\mathcal{F}$, and by Theorem \ref{ta} there exists $y_{\delta}\in X$ such that $D(\mathcal{F}_{\sigma_{i}}(y_{\delta}),x_{i})<L\delta$ for all $i\geq I_{\delta}$. By the same way, there exists $z_{\delta}\in X,$ Such that $D(\mathcal{F}_{\sigma_{-i}}(z_{\delta}),x_{-i})<L\delta$ for all $i\geq I_{\delta}$. Thus:$$D(\mathcal{F}_{\sigma_{i}}(y),\mathcal{F}_{\sigma_{i}}(y_{\delta}))\leq D(\mathcal{F}_{\sigma_{i}}(y),x_{i})+ D(x_{i}, \mathcal{F}_{\sigma_{i}}(y_{\delta}))<\epsilon_{0}$$ for all $i\geq I_{\delta}$. This implies that $\mathcal{F}_{\sigma_{I_{\delta}}}(y_{\delta})\in W_{\epsilon_{0}}^{s}(\mathcal{F}_{\sigma_{I_{\delta}}}(y))$. So that, by Lemma \ref{leg},  $D(\mathcal{F}_{\sigma_{i}}(y_{\delta}),\mathcal{F}_{\sigma_{i}}(y))\leq \eta^{i-I_{\delta}}D(\mathcal{F}_{\sigma_{I_{\delta}}}(y_{\delta}),\mathcal{F}_{\sigma_{I_{\delta}}}(y))$ for all $i\geq I_{\delta}$. Mimicking the procedure, we have $D(\mathcal{F}_{\sigma_{-i}}(y_{\delta}),\mathcal{F}_{\sigma_{-i}}(y))\leq \eta^{i-I_{\delta}}D(\mathcal{F}_{\sigma_{-I_{\delta}}}(y_{\delta}),\mathcal{F}_{\sigma_{-I_{\delta}}}(y))$ for all $i\geq I_{\delta}$. Take $J_{\delta}>I_{\delta}$ such that $\epsilon_{0}\eta^{i-I_{\delta}}<\delta $ if $i\geq J_{\delta}$. Since
$$D(\mathcal{F}_{\sigma_{i}}(y),x_{i})\leq D(\mathcal{F}_{\sigma_{i}}(y),\mathcal{F}_{\sigma_{i}}(y_{\delta}))+ D(\mathcal{F}_{\sigma_{i}}(y_{\delta}),x_{i})$$ and
$$D(\mathcal{F}_{\sigma_{-i}}(y),x_{-i})\leq D(\mathcal{F}_{\sigma_{-i}}(y),\mathcal{F}_{\sigma_{-i}}(y_{\delta}))+ D(\mathcal{F}_{\sigma_{-i}}(y_{\delta}),x_{-i}).$$
It is easy to see that \\
$\max\{D(\mathcal{F}_{\sigma_{i}}(y),x_{i}),D(\mathcal{F}_{\sigma_{-i}}(y),x_{-i})\}<(L+1)\delta$. Thus $\lim_{i\rightarrow\pm \infty}D(\mathcal{F}_{\sigma_{i}}(y),x_{i})=0$.
 \end{proof}
 Fix $\delta>0$ and $\lambda\in \Lambda$. Suppose that $\{x_{i}\}_{i \in \mathbb{Z_{+}}}$ is a $\delta-$pseudo orbit for $\mathcal{F}$ and consider $\{y_{i}\}_{i \in \mathbb{Z}}$ as the following:
 \[ y_{i} = \left\lbrace
  \begin{array}{c l}
 x_{i}& \text{if ~$~i\geq 0$},\\
    f^{i}_{\lambda}(x_{0}) & \text{if ~$~i< 0$}.
 \end{array}
  \right. \]
   So, $\{y_{i}\}_{i \in \mathbb{Z}}$ is a $\delta-$pseudo orbit for $\mathcal{F}$. Then shadowing properties on $\mathbb{Z}$ implies the shadowing properties on $\mathbb{Z_{+}}$. \\
   By Remark \ref{rea}, Theorems \ref{ta}, \ref{tb} and Theorem 2.2. of \cite{[GG]} we have the following corollary.
 \begin{corollary}\label{coa}
 If an \textbf{IFS} $\mathcal{F}=\{X; f_{\lambda}|\lambda\in\Lambda\} $ is uniformly expanding and if each function $f_{\lambda}(\lambda\in \Lambda) $ is homeomorphism, then the \textbf{IFS} has the Lipschitz shadowing and limit shadowing properties on $\mathbb{Z_{+}}$.
 \end{corollary}
 By Theorems \ref{ta}, \ref{tb} and Theorem 3.2. of \cite{[FN]} we have the following corollary.
 \begin{corollary}\label{cob}
Let $X$ be a compact metric space. If $ \mathcal{F}=\{X; f_{\lambda}|\lambda\in\Lambda\}$ is an expansive $IFS$ with the limit (Lipschitz) shadowing property ( on $\mathbb{Z}_{+}$), then so is $ \mathcal{F}^{-1}=\{X; g_{\lambda}|\lambda\in\Lambda\}$ where $f_{\lambda}:X\rightarrow X$ is homeomorphism and $g_{\lambda}=f^{-1}_{\lambda}$ for all $\lambda\in \Lambda$.
 \end{corollary}
  By Theorems \ref{ta}, \ref{tb} and Theorem 3.5. of \cite{[FN]} we have the following corollary.
 \begin{corollary}
 Let $\Lambda$ be a finite set, $ \mathcal{F}=\{X; f_{\lambda}|\lambda\in\Lambda\}$ is an $IFS$ and let $k>0$ be an integer. Set $\mathcal{F}^{k}= \{ g_{\mu}|\mu\in\Pi\}=\{f_{\lambda_{k}}o ...of_{\lambda_{1}}|\lambda_{1},...,\lambda_{k}\in\Lambda\}$. \\
If $\mathcal{F}$ has the limit (Lipschitz) shadowing property ( on $\mathbb{Z}_{+}$), then so does $ \mathcal{F}^{k}$.
 \end{corollary}
 We say that $\mathcal{F}$ is  strongly expansive if there exist a metric $d$ for $X$ and a constant $e\geq 0$ such that for every two arbitrary $\sigma,\mu\in \Lambda^{\mathbb{Z}}$,
 $d(\mathcal{F}_{\sigma_{i}}(x),\mathcal{F}_{\mu_{i}}(y))<e$, for all $i \in \mathbb{Z}$, implies that $x=y$.\\
 To prove Theorem \ref{tc}, we need the following lemma.
 \begin{lemma}\label{lea} Suppose that $\mathcal{F}$ is an expansive \textbf{IFS} with expansive constant $e$, $\alpha$ is an arbitrary positive number and $\sigma,~\mu\in \Lambda^{\mathbb{Z}}$. For all $x,y\in X$  there exists an integer \\$N=N(e,\alpha)>0$ such that if $d( \mathcal{F}_{\sigma_{i}}(x),\mathcal{F}_{\mu_{i}}(y))\leq e$ for all $\mid i\mid\leq N$, then $d(x,y)<\alpha$.\end{lemma}
\begin{proof}
We will give a proof by contradiction. Suppose that for each $n\geq 1$, there exist $x_{n}$ and $y_{n}$ with $d( \mathcal{F}_{\sigma_{i}}(x_{n}),\mathcal{F}_{\mu_{i}}(y_{n}))\leq e$ for all $\mid i\mid\leq n$ and $d(x_{n},y_{n})\geq\alpha$. Since $X$ is a compact metric space, we may assume that $x_{n}\rightarrow x$ and $y_{n}\rightarrow y$ and  hence that $\mathcal{F}_{\sigma_{i}}(x_{n})\rightarrow \mathcal{F}_{\sigma_{i}}(x)$ $\mathcal{F}_{\mu_{i}}(y_{n}))\rightarrow \mathcal{F}_{\mu_{i}}(y)$ for every $\mid i\mid\geq 1$. Then  $d( \mathcal{F}_{\sigma_{i}}(x),\mathcal{F}_{\mu_{i}}(y))\leq e$ for all $i\in \mathbb{Z}$ and $d(x,y)\geq \alpha$, which
is a contradiction.
\end{proof}
 \begin{theorem}\label{tc}
Let $X$ be a compact metric space and $\mathcal{F}=\{X; f_{\lambda}|\lambda\in\Lambda\} $ be an strongly  expansive IFS on $\mathbb{Z}$. The following conditions are equivalent:\\
 $i)~\mathcal{F}$ has the shadowing property,\\
 $ii)~\mathcal{F}$ has the continuous shadowing property.
 \end{theorem}
 \begin{proof}
 By definitions, this is clear that continuous shadowing property implies the shadowing property.\\
 $(i\Rightarrow ii)$. Let $e>0$ be an expansive constant of $ \mathcal{F}$. For $0<\epsilon <\frac{e}{3}$, let $\delta=\delta(\epsilon)<\epsilon$ be as in the shadowing property of $\mathcal{F}$. It is easy to see that for any $\delta-$pseudo orbit $\{x_{i}\}_{i\in \mathbb{Z}}$ of $\mathcal{F}$ by Remark \ref{remaa} there exists a unique $y\in X$ and $\sigma= \{...,\lambda_{-1},\lambda_{0},\lambda_{1},...\}\in \Lambda^{\mathbb{Z}}$ satisfying $d(\mathcal{F}_{\sigma_{i}}(y),x_{i})<\epsilon$ for all $i\in \mathbb{Z}$. So we define $r:P(\mathcal{F},\sigma)\rightarrow X$ by $r(\{x_{i}\}_{i\in \mathbb{Z}})=y$. To show that the map $r$ is continuous, choose an arbitrary constant $\alpha>0$, by Lemma \ref{lea} there exists an integer $N=N(e,\alpha)>0$ such that if $d( \mathcal{F}_{\sigma_{i}}(y),\mathcal{F}_{\mu_{i}}(y^{'}))\leq e$ for $\sigma,~\mu\in \Lambda^{\mathbb{Z}}$ and all $\mid i\mid\leq N$, then $d(y,y^{'})<\alpha$. Pick $\beta>0$ such that $2^{N}\beta<\frac{e}{3}$. Let $\{x_{i}\}_{i\in \mathbb{Z}}$, $\{x^{'}_{i}\}_{i\in \mathbb{Z}}\in P(\mathcal{F},\delta)$ be given two $\delta-$pseudo orbit of $\mathcal{F}$ with related sequences $\sigma,~\mu\in \Lambda^{\mathbb{Z}}$ and let $r(\{x_{i}\}_{i\in \mathbb{Z}}) =y$ and $r(\{x^{'}_{i}\}_{i\in \mathbb{Z}}) =y^{'}$.\\
if $d(\{x_{i}\}_{i\in \mathbb{Z}},\{x^{'}_{i}\}_{i\in \mathbb{Z}})<\beta$, then we see that $d(x_{i},x_{i}^{'})<\frac{e}{2}$ for all $\mid i\mid<N$, so that $d( \mathcal{F}_{\sigma_{i}}(y),\mathcal{F}_{\mu_{i}}(y^{'}))\leq d( \mathcal{F}_{\sigma_{i}}(y),x_{i})+d(x_{i},x^{'}_{i})+d( \mathcal{F}_{\mu_{i}}(y^{'}),x^{'}_{i})<e$
for all $\mid i\mid\leq N$. Thus $d(y,y^{'})<\alpha$ by the choice of $N$, and the conclusion is obtained.
\end{proof}

\end{document}